\documentclass[12pt]{amsart}
\usepackage{amssymb}
\usepackage{amsfonts}
\usepackage{amscd}
\usepackage{latexsym}
\usepackage{enumerate}
\usepackage{amsmath}
\usepackage{xy} \xyoption{all}
\numberwithin{equation}{section}

\newcounter{cs}
\stepcounter{cs}
\newcounter{ds}
\stepcounter{ds}

\newcommand{\casos}{\begin{itemize}}
\newcommand{\fcasos}{\end{itemize}\setcounter{cs}{1}}

\newcommand{\ol}{\overline}

\newcommand{\Si }{{\rm Sink}}

\newtheorem{lem}{Lemma}[section]

\newtheorem{theor}[lem]{Theorem}
\newtheorem{prop}[lem]{Proposition}

\theoremstyle{definition}
\newtheorem{defi}[lem]{Definition}

\newtheorem{remark}[lem]{Remark}

\newcommand{\Z}{\mathbb{Z}}

     \newcommand{\mon}[1]{\mathcal{V}(#1)}              %Monoide de projectius
     
     \newcommand{\cb}[0]{K}                             %Cos base

     \newcommand{\cam}[2][]{P_{#1}(#2)}
     \newcommand{\rat}[1]{P_{\mathrm{\!rat}}(#1)}
     \newcommand{\ser}[1]{P((#1))}

\begin{document}
\title[Mixed quiver algebras]{Mixed quiver algebras}
\author{Pere Ara and Miquel Brustenga}\address{Departament de Matem\`atiques, Universitat Aut\`onoma de
Barcelona, 08193, Bellaterra (Barcelona),
Spain}\email{para@mat.uab.es, mbrusten@mat.uab.es}

\thanks{Both authors were partially supported by
DGI MICIIN-FEDER MTM2008-06201-C02-01, and by the Comissionat per
Universitats i Recerca de la Generalitat de Catalunya. The second
author was partially supported by a grant of the Departament de
Matem\`{a}tiques, Universitat Aut\`{o}noma de Barcelona.}
\date{\today}

\begin{abstract}
In this paper we introduce a new class of $K$-algebras associated
with quivers. Given any finite chain $\mathbf{K}_r: K=K_0\subseteq
K_1\subseteq \cdots \subseteq K_r$ of fields and a  chain
$\mathbf{E}_r : H_0\subset H_1\subset \cdots \subset H_r=E^0$ of
hereditary saturated subsets of the set of vertices $E^0$ of a
quiver $E$, we build the mixed path algebra
$P_{\mathbf{K}_r}(E,\mathbf{H}_r)$, the mixed Leavitt path algebra
$L_{\mathbf{K}_r}(E,\mathbf{H}_r)$ and the mixed regular path
algebra $Q_{\mathbf{K}_r}(E,\mathbf{H}_r)$ and we show that they
share many properties with the unmixed species $P_K(E)$, $L_K(E)$
and $Q_K(E)$.
\end{abstract}
\maketitle

\section*{Introduction}

The work in the present paper is instrumental for the constructions
developed in \cite{Atams}, where the regular algebra of a finite
poset has been introduced in connection with the realization problem
for von Neumann regular rings, see also \cite{AB2}, \cite{directsum}
and \cite{Areal}. The reader is referred to these papers for further
information on the realization problem, and to \cite{AA1},
\cite{AA2}, \cite{ABC}, \cite{AB3} for related work on Leavitt path
algebras.

In the following, $\cb$ will denote a field and
$E=(E^0,E^1,r,s)$ a finite quiver (oriented graph) with
$E^0=\{1,\dotsc,d\}$. Here $s(e)$ is the {\em source vertex} of
the arrow $e$, and $r(e)$ is the {\em range vertex} of $e$. A {\em
path} in $E$ is either an ordered sequence of arrows
$\alpha=e_1\dotsb e_n$ with $r(e_t)=s(e_{t+1})$ for $1\leqslant
t<n$, or a path of length $0$ corresponding to a vertex $i\in
E^0$, which will be denoted by $p_i$. The paths $p_i$ are called
trivial paths, and we have $r(p_i)=s(p_i)=i$. A non-trivial path
$\alpha=e_1\dotsb e_n$ has length $n$ and we define
$s(\alpha)=s(e_1)$ and $r(\alpha)=r(e_n)$. We will denote the
length of a path $\alpha$ by $|\alpha|$, the set of all paths of
length $n$ by $E^n$, for $n>1$, and the set of all paths by $E^*$.

For $v,w\in E^0$, set $v\ge w$ in case there is a (directed) path
from $v$ to $w$. A subset $H$ of $E^0$ is called \emph{hereditary}
if $v\ge w$ and $v\in H$ imply $w\in H$. A set is \emph{saturated}
if every vertex which feeds into $H$ and only into $H$ is again in
$H$, that is, if $s^{-1}(v)\neq \emptyset$ and
$r(s^{-1}(v))\subseteq H$ imply $v\in H$. Denote by $\mathcal{H}$
(or by $\mathcal{H}_E$ when it is necessary to emphasize the
dependence on $E$) the set of hereditary saturated subsets of $E^0$.

Let us recall the construction from \cite{AB2} of the regular
algebra $Q_K(E)$ of a quiver $E$, although we will follow the
presentation in \cite{Atams} rather than the used in \cite{AB2}.
That is, relations (CK1) and (CK2) below are reversed with respect
to their counterparts in \cite{AB2}, so that we are led to work
primarily with {\it left} modules instead of right modules.

Therefore we recall the basic features of the regular algebra
$Q_K(E)$ in terms of the notation used here. We will only need
finite quivers in the present paper, so we restrict attention to
them.  The algebra $Q(E):=Q_K(E)$ fits into the following
commutative diagram of injective algebra morphisms:

\begin{equation}\notag \begin{CD}
K^d @>>> \cam{E} @>{\iota _{\Sigma}}>> \rat{E} @>>> \ser{E}\\
@VVV @V{\iota _{\Sigma_1}}VV @V{\iota _{\Sigma _1}}VV @V{\iota _{\Sigma _1}}VV\\
\cam{\ol{E}} @>>> L(E) @>{\iota _{\Sigma}}>> Q(E) @>>> U(E)
\end{CD} \notag \end{equation}

Here $P(E)$ is the path $K$-algebra of $E$, $\ol{E}$ denotes the
inverse quiver of $E$, that is, the quiver obtained by reversing the
orientation of all the arrows in $E$, $\ser{E}$ is the algebra of
formal power series on $E$, and $\rat{E}$ is the algebra of rational
series, which is by definition the division closure of $P(E)$ in
$P((E))$ (which agrees with the rational closure \cite[Observation
1.18]{AB2}). The maps $\iota_{\Sigma}$ and $\iota _{\Sigma_1}$
indicate universal localizations with respect to the sets $\Sigma$
and $\Sigma _1$ respectively. Here $\Sigma$ is the set of all square
matrices over $P(E)$ that are sent to invertible matrices by the
augmentation map $\epsilon \colon P(E)\to K^{|E^0|}$. By
\cite[Theorem 1.20]{AB2}, the algebra $\rat{E}$ coincides with the
universal localization $P(E)\Sigma ^{-1}$. The set
$\Sigma_1=\{\mu_v\mid v\in E^0,\,s^{-1}(v)\neq \emptyset\}$ is the
set of morphisms between finitely generated projective left
$P(E)$-modules defined by
 \begin{align*}
  \mu_v\colon P(E)v&\longrightarrow \bigoplus_{i=1}^{n_v}P(E)r(e^v_i)\\
  r&\longmapsto\left(re^v_1,\dotsc,re^v_{n_v}\right)
 \end{align*}
for any $v\in E^0$ such that $s^{-1}(v)\neq\emptyset$.
By a slight abuse of notation, we use also $\mu _v$ to denote the corresponding
maps between finitely generated projective left $\rat{E}$-modules and $\ser{E}$-modules
respectively.

The following relations hold in $Q(E)$:

The following relations hold in $Q(E)$:

(V)\, \, \, \,\hskip.35cm $p_vp_{v'}=\delta _{v,v'}p_v$  for all
$v,v'\in E^0$.

(E1) \, \, \hskip.3cm  $p_{s(e)}e=ep_{r(e)}=e$ for all $e\in E^1$.

(E2) \, \, \hskip.3cm $p_{r(e)}\ol{e}=\ol{e}p_{s(e)}=\ol{e}$ for all
$e\in E^1$.

(CK1)\hskip.5cm $\ol{e}e'=\delta _{e,e'}p_{r(e)}$ for all $e,e'\in
E^1$.

(CK2)\hskip.5cm  $p_v=\sum _{\{ e\in E^1\mid s(e)=v \}}e\ol{e}$ for
every $v\in E^0$ that emits edges.

The Leavitt path algebra $L(E)=P(E)\Sigma_1^{-1}$ is the algebra
generated by $\{p_v\mid v\in E^0\}\cup \{e,\ol{e}\mid e\in E ^1\}$
subject to the relations (1)--(5) above. By \cite[Theorem 4.2]{AB2},
the algebra $Q(E)$ is a von Neumann regular hereditary ring and
$Q(E)=P(E)(\Sigma \cup \Sigma_1)^{-1}$. Here the set $\Sigma$ can be
clearly replaced with the set of all square matrices of the form
$I_n+B$ with $B\in M_n(P(E))$ satisfying $\epsilon (B)=0$, for all
$n\ge 1$.

\section{Structure of ideals}
\label{sect:strucideals}

The structure of the lattice of ideals of $Q(E)$ can be neatly
computed from the graph. Let $H$ be a hereditary saturated subset of
$E^0$. Define the graph $E/H$ by $(E/H)^0=E^0\setminus H$ and
$(E/H)^1=\{e\in E^1: r(e)\notin H\}$, with the functions $r$ and $s$
inherited from $E$. We also define $E_H$ as the restriction of the
graph $E$ to $H$, that is $(E_H) ^0=H$ and $(E_H)^1=\{e\in E^1:
s(e)\in H\}$. For $Y\subseteq E^0$ set $p_Y=\sum _{v\in Y} p_v$.

\begin{prop}
\label{quotient}

(a) The ideals of $Q(E)$ are in one-to-one correspondence with the
order-ideals of $M_E$ and consequently with the hereditary and
saturated subsets of $E$.

(b) If $H$ is a hereditary saturated subset of $E$, then
$Q(E)/I(H)\cong Q(E/H)$, where $I(H)$ is the ideal of $Q(E)$
generated by the idempotents $p_v$ with $v\in H$.

(c) Let $H$ be a hereditary subset of $E^0$. Then the following properties hold:
\begin{enumerate}
\item $P(E_H)=p_HP(E)=p_HP(E)p_H$,
\item $P((E_H))=p_HP((E))=p_HP((E))p_H$,
\item $P_{{\rm rat}}(E_H)=p_HP_{{\rm rat}}(E)=p_HP_{{\rm rat}}(E)p_H$,
\item $Q(E_H)\cong p_HQ(E)p_H$.
\end{enumerate}
\end{prop}

\begin{proof}
(a) By \cite[Theorem 4.2]{AB2} we have a monoid isomorphism
$\mathcal V (Q(E))\cong M_E$. Since $Q(E)$ is von Neumann regular,
we have a lattice isomorphism $L_2(Q(E))\cong L(M_E)$, where
$L_2(Q(E))$ denotes the lattice of two-sided ideals of $Q(E)$ and
$L(M_E)$ denotes the lattice of order-ideals of $M_E$,
 cf. \cite[Proposition 7.3]{GW}. Now by \cite[Proposition 5.2]{AMP}
there is a lattice isomorphism $L(M_E)\cong \mathcal H$, where
$\mathcal H$ is the lattice of hereditary saturated subsets of
$E^0$. Given an ideal $I$ of $Q(E)$, the set of vertices $v$ such
that $p_v\in I$ is a hereditary saturated subset of $E^0$ which
generates $I$ as an ideal.

(b) We shall use some universal properties. Let $H$ be a hereditary
saturated subset of $E^0$ and let $I(H)$ be the ideal of $Q(E)$
generated by $H$. By \cite[Lemma 2.3]{APS}, there is a $K$-algebra
isomorphism $\varphi \colon L(E)/J\to L(E/H)$, where $J$ is the
ideal of $L(E)$ generated by the idempotents $p_v$ with $v\in H$.
The isomorphism $\varphi$ is defined in such a way that it is the
identity on $(E/H)^*$ and $0$ on $E^*\setminus (E/H)^*$. Write
$\Sigma(E)$ (resp. $\Sigma(E/H)$) for the set of matrices of the
form $I_n+B$, where $B\in M_n(P(E))$ (resp. $B\in M_n(P(E/H))$)
satisfies $\epsilon (B)=0$. Clearly $\varphi (\Sigma (E))\subseteq
\Sigma (E/H)$ so that the map
$$\widetilde{\varphi}=\varphi \circ \pi \colon L(E)\to L(E)/J\to
L(E/H)$$ gives rise to an algebra homomorphism $Q(E)\to Q(E/H)$
which is $0$ on $H$, so we get a homomorphism $\rho \colon
Q(E)/I(H)\to Q(E/H)$.

To construct the inverse, consider the map $\psi \colon L(E/H)\to
Q(E)/I(H) $ which is given by the composition of $\varphi
^{-1}\colon L(E/H)\to L(E)/J $ and the natural map $L(E)/J\to
Q(E)/I(H)$. Clearly $\psi (\Sigma (E/H))$ is contained in the set
of invertible matrices over $Q(E)/I(H)$, because each element in
$\Sigma (E/H)$ can be lifted to an element in $\Sigma (E)$. It
follows from the universal property of $Q(E/H)$ that there is a
unique homomorphism $\lambda \colon Q(E/H)\to Q(E)/I(H) $
extending $\psi$. Using uniqueness of extensions, it is fairly
easy to see that $\lambda \circ \rho =\text{Id}_{Q(E)/I(H)}$ and
$\rho \circ \lambda=\text{Id}_{Q(E/H)}$.

(c) (1), (2): This is clear from the fact that $H$ is a
hereditary subset of $E^0$.

(3) The algebra $p_HP_{\text{rat}}(E)=p_HP_{\text{rat}}(E)p_H$ is
rationally closed in $p_HP((E))p_H=P((E_H))$ and contains $P(E_H)$,
so that $P_{\text{rat}}(E_H)\subseteq p_HP_{\text{rat}}(E)$.

It remains to show that $ p_HP_{\text{rat}}(E)\subseteq
P_{\text{rat}}(E_H)$. If $a\in P_{\text{rat}}(E)$, there exist by
\cite[Theorem 7.1.2]{free} a row $\gamma \in {^nP(E)}$, a column
$\delta \in P(E)^n$ and a matrix $B\in M_n(P(E))$ such that
$\epsilon (B)=0$ such that
$$a=\gamma (I-B)^{-1}\delta .$$ Note that, since $H$ is a hereditary
subset of $E^0$, we have $p_H\tau =p_H\tau p_H$ for every matrix
$\tau$ over $P((E))$. Applying this we get
$$p_Ha=(p_H\gamma  p_H)(p_HI_n- (p_HBp_H))^{-1}(p_H\delta p_H) ,$$
which shows that $p_Ha=p_Hap_H\in P_{\text{rat}}(E_H)$.

(4) We have a map $P(E_H)=p_HP(E)p_H\to p_HQ(E)p_H$ which is clearly
$(\Sigma (E_H)\cup \Sigma _1(E_H))^{-1}$-inverting and thus induces
a $K$-algebra homomorphism $Q(E_H)\to p_HQ(E)p_H$. Since this map
does not annihilate any basic idempotent $p_v$, we conclude from (a)
that it is injective, so that we can consider $Q(E_H)$ as a
subalgebra of $p_HQ(E)p_H$.

To show the reverse containment, recall from \cite{AB2} that an
element $a\in Q(E)$ can be written as a finite sum
$$a=\sum_{\gamma\in E^*} a_{\gamma}\ol{\gamma} ,$$
where $a_{\gamma}\in \rat{E}p_{s(\gamma)}$. We get
$$p_Hap_H= \sum_{\gamma\in (E_H)^*} (p_Ha_{\gamma}p_H)\ol{\gamma}$$
with $p_Ha_{\gamma}p_H=p_Ha_{\gamma}\in
p_HP_{\text{rat}}(E)=P_{\text{rat}}(E_H)$ by (c). Thus $p_Hap_H\in
Q(E_H)$ as desired.
\end{proof}

\section{Mixed quiver algebras}

Since we will be playing in this section with different fields, it
will be convenient that our notation remembers the field we are
considering, henceforth we will denote the path $K$-algebra by
$P_K(E)$, the regular $K$-algebra of the quiver by $Q_K(E)$, and so
on.

Let $K\subseteq L$ be a field extension and let $E$ be a finite
quiver. There is an obvious $K$-algebra homomorphism $h\colon
Q_K(E)\to Q_L(E)$ which satisfies $h(p_v)\ne 0$ for all $v\in E^0$.
It follows from Proposition \ref{quotient} that the map $h$ is
injective. Using this map, we will view $Q_K(E)$ as a $K$-subalgebra
of $Q_L(E)$. Let $H$ be a hereditary saturated subset of $E^0$ and
consider the idempotent
$$p_H=\sum _{v\in H}p_v\in Q_K(E)\subseteq Q_L(E).$$
By Proposition \ref{quotient}(c)(4) we have that $p_HQ_L(E)p_H\cong Q_L(E_H)$,
where $E_H$ denotes the
restriction of $E$ to $H$. The {\it mixed
regular path algebra} $Q_{K\subseteq L}(E,H)$ is defined as the
$K$-subalgebra of $Q_L(E)$ generated by $Q_K(E)$ and
$p_HQ_L(E)p_H$. Observe that
$$Q_{K\subseteq L}(E,H)=Q_K(E)+Q_K(E)(p_HQ_L(E)p_H)Q_K(E)$$
and that $I=Q_K(E)(p_HQ_L(E)p_H)Q_K(E)$ is an ideal in
$Q=Q_{K\subseteq L}(E,H)$ such that $Q/I\cong Q_K(E/H)$, because
$I\cap Q_K(E)$ agrees with the ideal $I_K(H)$ of $Q_K(E)$
generated by $H$.

\begin{defi}\label{buildblock}
Let $K_0\subseteq K_1\subseteq \cdots \subseteq K_r$ be a chain of
fields. Let $E$ be a finite quiver and let $H_0\subset H_1\subset
\cdots \subset H_r=E^0$ be a chain of hereditary saturated subsets
of $E^0$. We build rings $R_i$ $i=0,1,\dots ,r$ inductively as follows:

{\rm (1)} $R_0=Q_{K_r}(E_{H_0})$.

{\rm (2)} $R_i=Q_{K_{r-i}}(E_{H_i})+
Q_{K_{r-i}}(E_{H_i})p_{H_{i-1}}R_{i-1}p_{H_{i-1}}Q_{K_{r-i}}(E_{H_i}) $ for $1\le i\le r$.

Each $R_i$ is a unital $K_{r-i}$-algebra with unit $p_{H_i}$ and
we have $Q_{K_{r-i}}(E_{H_i})\subseteq R_i\subseteq
Q_{K_r}(E_{H_i})$.
\end{defi}

Before we establish the basic properties of our construction, we
simplify notation as follows. A chain of fields of length $r$ will
be denoted:
$$\mathbf{K}_r : K_0\subseteq K_1 \subseteq \cdots \subseteq K_r .$$
(Note that the inclusions need not be strict.)
 Similarly a chain of hereditary saturated subsets of $E^ 0$ of length $r$ will be
 denoted:
 $$\mathbf{H}_r : H_0\subset H_1\subset \cdots \subset H_r=E^0 .$$
 (Here we have strict inclusions. The choice of strict/non-strict inclusions is
 made to gain flexibility in the notation, and in particular with regard to be
 aligned with the notation used in
 \cite{Atams}.)
 Now we denote the $K_0$-algebra $R_r$
 constructed in Definition \ref{buildblock} by $Q_{\mathbf{K}_r}(E;\mathbf{H}_r)$.
The straightforward proof of the next two results is left to the
reader.

\begin{prop}\label{cutting}
Let $\mathbf{K}_r$, $\mathbf{H}_r$ and $Q_{\mathbf{K}_r}(E;\mathbf{H}_r)$
be as before. Let $I_{i-1}$ be the ideal of $Q_{\mathbf{K}_r}(E;\mathbf{H}_r)$
generated by $p_{H_{i-1}}$. Then
$$Q_{\mathbf{K}_r}(E;\mathbf{H}_r)/I_{i-1}\cong Q_{\mathbf{K}_{r-i}}(E/H_{i-1};
\mathbf{H}^{r-i}) ,$$
where
$$\mathbf{K}_{r-i} : K_0\subseteq K_1 \subseteq \cdots \subseteq K_{r-i} $$
and
$$\mathbf{H}^{r-i}:  H_i\setminus H_{i-1}\subset H_{i+1}\setminus H_{i-1}
\subset \cdots \subset H_r\setminus H_{i-1}=(E/H_{i-1})^0 .$$
\end{prop}

\begin{prop}\label{cornering}
Let $\mathbf{K}_r$,  $\mathbf{H}_r$ and
$Q_{\mathbf{K}_r}(E;\mathbf{H}_r)$ be as before. Then
$$p_{H_i}Q_{\mathbf{K}_r}(E;\mathbf{H}_r)p_{H_i}\cong Q_{\mathbf{K}^{i}}(E_{H_i};
\mathbf{H}_i) ,$$ where
$$\mathbf{K}^{i} : K_{r-i}\subseteq K_{r-i+1} \subseteq \cdots \subseteq K_{r} $$
and
$$\mathbf{H}_{i}:  H_0\subset H_1
\subset \cdots \subset H_i .$$
\end{prop}

We are going to show that the algebras
$Q_{\mathbf{K}_r}(E;\mathbf{H}_r)$ above are universal localizations
of suitable mixed path algebras. This is analogous to the situation
with the usual path algebra of a quiver and its regular algebra
\cite{AB2}, and plays an important role in the applications, see
\cite[Sections 5 and 6]{Atams}.

We retain the above notation. The {\it mixed path algebra}
$P_{\mathbf{K}_r}(E;\mathbf{H}_r)$ is the $K_0$-subalgebra of the
usual path $K_r$-algebra $P_{K_r}(E)$ defined inductively as
follows. Set $P_0:=P_{K_r}(E_{H_0})$, and for $1\le i\le r$, put
$P_i:=P_{K_{r-i}}(E_{H_i})+ P_{K_{r-i}}(E_{H_i})p_{H_{i-1}}P_{i-1}$.
Then the $K_0$-algebra $P_{\mathbf{K}_r}(E;\mathbf{H}_r)$ is by
definition the algebra $P_r$. Observe that this algebra is the usual
path algebra whenever all the fields in the chain are equal.

Assume that $|E^0|=d$. The usual augmentation $\epsilon \colon
P_{K_r}(E)\to K_r^d$ restricts to a surjective split homomorphism
$$\epsilon\colon P_{\mathbf{K}_r}(E;\mathbf{H}_r)\longrightarrow
\prod _{i=0}^r\,\,\prod_{v\in H_{i}\setminus H_{i-1}} K_{r-i}p_v.$$
Similar definitions give the mixed power series algebra over the
quiver $P_{\mathbf{K}_r}((E;\mathbf{H}_r))$ and the mixed algebra of
rational power series $P^{\rm rat}_{\mathbf{K}_r}(E;\mathbf{H}_r)$.
For instance, when $r=1$ we have $P^{\rm
rat}_{\mathbf{K}_1}(E;\mathbf{H}_1)=P^{\rm rat}_{K_0}(E)+P^{\rm
rat}_{K_0}(E)p_{H_0}P^{\rm rat}_{K_1}(E_{H_0})$.

The following generalizes the unmixed case \cite[Theorem 1.20]{AB2}.

\begin{theor}
\label{rat-locali} Let $\mathbf{K}_r$, $\mathbf{H}_r$ and
$P_{\mathbf{K}_r}(E;\mathbf{H}_r)$ be as before. Let $\Sigma$ denote
the set of matrices over $P_{\mathbf{K}_r}(E;\mathbf{H}_r)$ that are
sent to invertible matrices by $\epsilon$. Then $P^{\rm
rat}_{\mathbf{K}_r}(E;\mathbf{H}_r)$ is the rational closure of
$P_{\mathbf{K}_r}(E;\mathbf{H}_r)$ in $P_{K_r}((E))$, and the
natural map $P_{\mathbf{K}_r}(E;\mathbf{H}_r)\Sigma ^{-1}\to P^{\rm
rat}_{\mathbf{K}_r}(E;\mathbf{H}_r)$ is an isomorphism. \end{theor}

\begin{proof}
We will give the proof in the case $r=1$. An easy induction
argument can be used to get the general case.

So assume that we have a field extension $K\subseteq L$ and a hereditary
saturated subset $H$ of $E^0$. We have to show that $S:=P_K^{\rm rat}(E)+
P_K^{\rm rat}(E)p_HP_L^{\rm rat}(E_H)$ is the rational closure of
$R:=P_K(E)+
P_K(E)p_HP_L(E_H)$ in $P_L((E))$, the algebra of power series over $E$ with
coefficients in $L$. Write $\mathcal R$ for this rational closure.

We start by showing that $S\subseteq \mathcal R$. Since $P^{\rm rat}_K(E)$
is the rational closure of $P_K(E)$ inside $P_L((E))$, we see that
$P^{\rm rat}_K(E)\subseteq \mathcal R$.
Also, note that the algebra $p_H\mathcal R=p_H\mathcal R p_H$
is inversion closed in $p_HP_L((E_H))$ and contains $p_HP_L(E_H)$,
so it must contain the rational closure
of $p_HP_L(E_H)$ in $p_HP_L((E_H))$ which is precisely $p_HP_L^{\rm rat}(E_H)$.
It follows that $P^{\rm rat}_K(E)$ and $p_HP_L^{\rm rat}(E_H)$ are both contained in
$\mathcal R$. Since $\mathcal R$ is a ring, we get $S\subseteq \mathcal R$.

To show the reverse inclusion $\mathcal R \subseteq S$, take any
element $a$ in $\mathcal R$. There exist a row $\lambda\in {^nR}$, a
column $\rho\in R^n$ and a matrix $B\in M_n(R)$ such that $\epsilon
(B)=0$ such that
\begin{equation}
\label{equ:5.3}
a=\lambda (I-B)^{-1}\rho .
\end{equation}
Now the matrix $B$ can be written as $B=B_1+B_2$, where $B_1\in P_K(E)\subseteq R$
and $B_2\in R$ satisfy that $\epsilon (B_1)=\epsilon (B_2)=0$, all the entries
of $B_1$ are supported on paths ending in $E^0\setminus H$ and all the
entries of $B_2$ are supported on paths ending in $H$. Note that, since $H$
is hereditary, this implies that all the paths in the support of the
entries of $B_1$ start in $E^0\setminus H$ and thus $B_2B_1=0$.
It follows that
\begin{equation}
\label{Binverse}
(I-B)^{-1}=(I-B_1-B_2)^{-1}=(I-B_1)^{-1}(I-B_2)^{-1},
\end{equation}
and therefore
$(I-B)^{-1}=(I-B_1)^{-1}+(I-B_1)^{-1}B_2(I-B_2)^{-1}\in M_n(S).$
It follows from (\ref{equ:5.3}) that $a\in S$, as desired.

Since the set $\Sigma$ is precisely the set of square matrices over
$R$ which are invertible over $P_L((E))$, we get from a well-known general result
(see for instance \cite[Lemma
10.35(3)]{Luck}) that there is a surjective $K$-algebra homomorphism $\phi \colon R{\Sigma}^{-1}\to \mathcal R$.

The rest of the proof is devoted to show that $\phi $ is injective.
We have a commutative diagram
\begin{equation}
\begin{CD}
P_K(E)\Sigma(\epsilon_K )^{-1} @>>> R\Sigma^ {-1} @>>> P_L(E)\Sigma (\epsilon _L)^{-1} \\
@V{\phi_K}V{\cong}V  @V{\phi}VV   @V{\phi_L}V{\cong}V \\
P_K^{\text{rat}}(E) @>>> \mathcal R @>>> P_L^{\text{rat}}(E)
\end{CD}
\end{equation}
The map $P_K(E)\Sigma(\epsilon_K )^{-1} \to  P_L(E)\Sigma (\epsilon _L)^{-1} $ is injective,
so the map $P_K(E)\Sigma(\epsilon_K )^{-1} \to R\Sigma^ {-1}  $ must also be injective.
Hence the $K$-subalgebra of $R\Sigma^{-1}$ generated by $P_K(E)$ and the entries of the inverses of matrices in $\Sigma (\epsilon_K )$
is isomorphic to $P_K^{\text{rat}}(E)$. Observe that we can replace $\Sigma$ by the set of matrices of the form $I-B$, where $B$ is a square matrix over $R$
with $\epsilon (B)=0$. As before we write $B=B_1+B_2$, where all the entries of $B_1$ end in $E^0\setminus H$ and all the entries in $B_2$ end in $H$, and thus $B_2B_1=0$,
so that (\ref{Binverse}) holds in $R\Sigma^{-1}$.
An element $x$ in $R\Sigma^{-1}$ is of the form
\begin{equation}
\label{can-form}
x=\lambda (I-B)^{-1} \rho
\end{equation}
with $\lambda \in {^n R}$ and $\rho \in R^ n$,  and $\epsilon (B)=0$.

\medskip

\noindent {\it Claim 1.}  We have
$$p_HR\Sigma ^{-1} =p_HP_L^{\text{rat}}(E_H)= p_H P_L(E_H)\Sigma (\epsilon _L^H)^{-1}p_H .$$

\noindent {\it Proof of Claim 1.}
Observe first that we have a natural $L$-algebra homomorphism $P_L(E_H)\Sigma (\epsilon _L^ H)^{-1}\to p_HR\Sigma^{-1}$.
The composition of this map with the map $R\Sigma^ {-1} \to P_L(E)\Sigma (\epsilon _L)^{-1}$ is injective (since its image is
 $p_HP_L^{\text{rat}}(E_H)\cong
P_L^{\text{rat}}(E_H)\cong P_L(E_H)\Sigma (\epsilon _L^ H)^{-1}$) so
the map $P_L(E_H)\Sigma (\epsilon _L^ H)^{-1}\to p_HR\Sigma^{ -1}$
must be injective. We identify $p_HP_L^{\text{rat}}(E_H)$ with its
image in $p_HR\Sigma^{-1}$, which is the $L$-subalgebra of
$p_HR\Sigma^{-1}$ generated by $p_HP_L(E_H)$ and the entries of the
inverses of matrices of the form $p_HI-B$, with $B$ a square matrix
over $p_HP_L(E_H)$  with $\epsilon (B)=0$. For an element $x$ in
$R\Sigma ^{-1}$, we write it in its canonical form (\ref{can-form})
and we write $B=B_1+B_2$ with all the entries in $B_1$ ending in
$E^0\setminus H$ and all the entries of $B_2$ ending in $H$.

Now multiply (\ref{can-form}) on the left by $p_H$ and use (\ref{Binverse}) to get
\begin{align*}
p_Hx & =p_H\lambda (I-B_1)^{-1}\rho +p_H\lambda (I-B_1)^{ -1}B_2(I-B_2)^{-1}\rho   \\
& =   p_H\lambda p_H(I-B_1)^{-1}\rho + p_H \lambda p_H(I-B_1)^{ -1}
B_2(I-B_2)^{-1}\rho \\
& = p_H \lambda p_H \rho + p_H \lambda p_H B_2 p_H (I-B_2)^{-1} \rho .
\end{align*}
 Write $B_2=B_2'+B_2''$, where all the entries of $B_2'$ start  in $E^0\setminus H$ and all the entries in $B_2''$ start in $H$ (and so end in $H$ as well).
 Note that $(I-B_2')^{-1}=I+B_2'$, because $B_2'^ 2=0$, so that $p_H(I-B_2')^{-1} =p_H$. Since $B_2''B_2'=0$ we have  $(I-B_2)^{-1}=(I-B_2')^{-1}(I-B_2'')^{-1}$,
 and thus
 $$p_Hx=p_H\lambda p_H\rho p_H + p_H\lambda p_H B_2 p_H (I-B_2'')^{-1} p_H\rho p_H .$$
 It follows that $p_Hx\in p_H P_L^{\text{rat}} (E_H)$, as wanted. \qed

 Assume now that $x\in \ker (R\Sigma ^{-1} \to \mathcal R )= \ker ( R\Sigma ^{-1} \to P_L(E)\Sigma (\epsilon _L)^{-1})$ and write $x$ as in (\ref{can-form}), with $B=B_1+B_2$ as before.
 Then
\begin{equation}
\label{express1}
x=\lambda (I-B_1)^{-1}\rho + \lambda (I-B_1)^{-1} B_2(I-B_2)^{-1} \rho .
\end{equation}
Multiplying on the right by $1-p_H$, we get
 $$x(1-p_H)=\lambda (I-B_1)^{-1}\rho (1-p_H) = \lambda (1-p_H)(I-B_1)^{-1}\rho (1-p_H)  \in P_K(E)\Sigma (\epsilon _K)^{-1}$$
 and $0=\phi (x(1-p_H))=\phi_K(x(1-p_H))$.  Since $\phi _K$ is an isomorphism, we get $x(1-p_H)=0$.

Hence we have
\begin{equation}
\label{express2} x=  \lambda (I-B_1)^{ -1} \rho_2 +  \lambda
(I-B_1)^{-1}B_2(I-B_2)^{-1}\rho_2   ,
\end{equation}
 where $\rho=\rho_1+\rho_2$ with $\rho_1$ ending in $E^0\setminus H$ and $\rho_2$ ending in $H$.
 By Claim 1 we have $p_Hx=0$, because $\phi $ is an isomorphism when restricted to $p_HP_L^{\text{rat}}(E_H)$.
 Now we are going to find a suitable expression for $x=(1-p_H)xp_H$.
 Write $\lambda =\lambda _1+\lambda _2$ with $\lambda _1=(1-p_H)\lambda $ and $\lambda _2=p_H\lambda $.
  Then
  \begin{equation}
 \label{express3}
(1-p_H) \lambda (I-B_1)^{-1}\rho_2 =\lambda _1 (I-B_1)^{ -1}\rho _2 .
 \end{equation}
 Similarly $(1-p_H)\lambda (I-B_1)^{-1}B_2(I-B_2)^{-1}\rho_2 =\lambda _1 (I-B_1)^{-1} B_2(I-B_2)^{-1}\rho_2 $.
 Write $B_2=B_2'+B_2''$, with $B_2'$ starting in $E^0\setminus H$ and $B_2''$ starting in $H$. Then $B_2''B_2'=0$ and $(I-B_2)^{-1}=(I-B_2')^{-1}(I-B_2'')^{-1}$,
 so that
  \begin{align}
 \label{express4}
\notag &  (1-p_H)\lambda (I-B_1)^{-1}B_2(I-B_2)^{-1}\rho_2 =\lambda _1 (I-B_1)^{-1} B_2(I-B_2)^{-1}\rho_2 \\
  &  = \lambda_1 (I-B_1)^{-1}B_2(I+B_2')(I-B_2'')^{-1}\rho_2 \\
 \notag & =  \lambda_1 (I-B_1)^{-1}B_2(I-B_2'')^{-1}\rho_2.
   \end{align}
 Substituting (\ref{express3}) and (\ref{express4}) in (\ref{express2}) we get
   \begin{equation}
\label{express5}
x=  (1-p_H)xp_H=  \lambda_1 (I-B_1)^{ -1} \rho_2+  \lambda_1(I-B_1)^{-1}B_2  (I-B_2'')^{-1}\rho_2   .
\end{equation}
It follows that $x\in \sum _{i=1}^k P_K^{\text{rat}}
(E/H)e_iP_L^{\text{rat}} (E_H)$, where $e_1,\dots ,e_k$ is the
family of {\it crossing edges}, that is, the family of edges $e\in
E^1$ such that $s(e)\in E^0\setminus H$ and $r(e)\in H$. Write
$x=\sum _{i=1}^k \sum _{j=1}^{m_i} a_{ij}e_ib_{ij} $ for certain
$a_{ij}\in P_K^{\text{rat}} (E/H)$ and $b_{ij}\in P_L^{\text{rat}}
(E_H)$. Then we have
$$0=\phi (x)= \sum _{i=1}^k \sum _{j=1}^{m_i} a_{ij}e_ib_{ij} ,$$
this element being now in $P_L((E))$. Clearly this implies that $ \sum _{j=1}^{m_i} a_{ij}e_ib_{ij}=0$ in $P_L((E))$ for all $i= 1,\dots ,k$.
So the result follows from the following claim:

\medskip

\noindent  {\it Claim 2.}  Let $e$ be a crossing edge, so that $s(e)\in E^0\setminus H$ and $r(e)\in H$.
Assume that $b_1,\dots ,b_m\in p_{r(e)}P_L((E_H))$ are $K$-linearly independent elements, and assume that $a_1e,\dots , a_me$ are not all $0$, where
$a_1,\dots ,a_m\in P_K((E\setminus H))$. Then $\sum _{i=1}^m a_ieb_i \ne 0$ in $P_L((E))$.

\noindent {\it Proof of Claim 2.} By way of contradiction, suppose
that $\sum _{i=1}^m a_ieb_i = 0$. We may assume that $a_1e\ne 0$.
Let $\gamma $ be a path in the support of $a_1$ such that $r(\gamma
)=s(e)$. For every path $\mu$ with $s(\mu )=r(e)$ we have that the
coefficient of $\gamma e \mu $ in $a_ieb_i$ is $a_i(\gamma)
b_i(\mu)$, so that $\sum _{i=1}^m a_i(\gamma )b_i(\mu) =0$ for every
$\mu$ such that $s(\mu)=r(e)$. Since every path in the support of
each $b_i$ starts with $r(e)$, we get that
$$\sum_{i=1}^m a_i(\gamma )b_i=0$$
with $a_1(\gamma)\ne 0$, which contradicts the linear independence over $K$ of $b_1,\dots, b_m$. \qed

This concludes the proof of the theorem.
\end{proof}

Following \cite[Section 2]{AB2}, we define, for $e\in E^1$, the right transduction
$\tilde{\delta}_e\colon P_L((E))\to P_L((E))$ corresponding to $e$ by $$\tilde{\delta_e} (\sum_{\alpha\in
E^*}\lambda_\alpha \alpha)=\sum_{\substack{\alpha\in
E^*\\s(\alpha)=r(e)}} \lambda_{e\alpha } \alpha.$$
Similarly the left transduction corresponding to $e$ is given by
$$\delta_e (\sum_{\alpha\in
E^*}\lambda_\alpha \alpha)=\sum_{\substack{\alpha\in
E^*\\r(\alpha)=s(e)}} \lambda_{\alpha e} \alpha.$$

Observe that $R:=P_{\mathbf{K}_r}(E;\mathbf{H}_r)$ is closed under all the right transductions, i.e.  $\tilde{\delta}_e(R)\subseteq R$,
but  $R$ is not invariant under all the left transductions. Some of the proofs in \cite{AB2}
make use of the fact that the usual path algebra $P_K(E)$ is closed under {\it left and right}  transductions.
Fortunately we have been able to overcome the potential problems arising from the failure of invariance of $R$ under left transductions
by using alternative arguments.

We are now ready to get a description of the algebra $Q_{\mathbf{K}_r}(E;\mathbf{H}_r)$
as a universal localization of the mixed path algebra $P_{\mathbf{K}_r}(E;\mathbf{H}_r)$.

 Write $R:=P_{\mathbf{K}_r}(E;\mathbf{H}_r)$. For any $v\in E^0$ such that $s^{-1}(v)\neq\emptyset$ we put
 $s^{-1}(v)=\{e^v_1,\dotsc,e^v_{n_v}\}$, and we
consider the left $R$-module homomorphism
 \begin{align*}
  \mu_v\colon Rv&\longrightarrow \bigoplus_{i=1}^{n_v}Rr(e^v_i)\\
  r&\longmapsto\left(re^v_1,\dotsc,re^v_{n_v}\right)
 \end{align*}
 Write $\Sigma_1=\{\mu_v\mid v\in E^0,\,s^{-1}(v)\neq \emptyset\}$.

\begin{theor}
\label{hereditarymixed} Let $\mathbf{K}_r$ and $\mathbf{H}_r$ and
$P_{\mathbf{K}_r}(E;\mathbf{H}_r)$ be as before. Let $\Sigma$ denote
the set of matrices over $P_{\mathbf{K}_r}(E;\mathbf{H}_r)$ that are
sent to invertible matrices by $\epsilon$ and let $\Sigma _1$ be the
set of maps defined above. Then we have
$Q_{\mathbf{K}_r}(E;\mathbf{H}_r)=
(P_{\mathbf{K}_r}(E;\mathbf{H}_r))(\Sigma \cup \Sigma _1)^{-1}$.
Moreover $Q_{\mathbf{K}_r}(E;\mathbf{H}_r)$ is a hereditary von
Neumann regular ring and all finitely generated projective
$Q_{\mathbf{K}_r}(E;\mathbf{H}_r)$-modules are induced from $P^{{\rm
rat}}_{\mathbf{K}_r}(E;\mathbf{H}_r)$.
\end{theor}

\begin{proof} First observe that the mixed path algebra $P_{\mathbf{K}_r}(E;\mathbf{H}_r)$
is a hereditary ring and that
$\mon{P_{\mathbf{K}_r}(E;\mathbf{H}_r)}=(\mathbb Z^+)^d$, where
$|E^0|=d$. This follows by successive use of \cite[Theorem
5.3]{Bergman}.

In order to get that the right transduction $\tilde{\delta}_e\colon
P_{K_r}((E))\to P_{K_r}((E))$ corresponding to $e$ is a right
$\tau_e$-derivation on $P_{\mathbf{K}_r}(E;\mathbf{H}_r)$, that is,
\begin{equation}
\label{right-der} \tilde{\delta}_e
(rs)=\tilde{\delta}_e(r)s+\tau_e(r)\tilde{\delta}_e(s)
\end{equation}
for all $r,s\in P_{\mathbf{K}_r}(E;\mathbf{H}_r)$, we have to modify
slightly the definition of $\tau _e$ given in \cite[page 220]{AB2}.
Concretely we define $\tau _e$ as the endomorphism of $P_{K_r}((E))$
given by the composition
$$P_{K_r}((E))\to \prod_{v\in E^0} K_rp_v\to \prod_{v\in E^0} K_rp_v\to P_{K_r}((E)),$$
where the first and third maps are the canonical projection and
inclusion respectively, and the middle map is the $K_r$-lineal map
given by sending $p_{s(e)}$ to $p_{r(e)}$, and any other idempotent
$p_v$ with $v\ne s(e)$ to $0$. Observe that this restricts to an
endomorphism of $P_{\mathbf{K}_r}(E;\mathbf{H}_r)$ and that the
proof in \cite[Lemma 2.4]{AB2} gives the desired formula
(\ref{right-der}) for $r,s\in P_{K_r}((E))$ and, in particular for
$r,s\in P_{\mathbf{K}_r}(E;\mathbf{H}_r)$. The constructions in
\cite[Section 2]{AB2} apply to $R:=P^{\rm
rat}_{\mathbf{K}_r}(E;\mathbf{H}_r)$ (with some minor changes), and
we get that $R\Sigma_1^{-1}=R\langle \ol{E};
\tau,\tilde{\delta}\rangle /I$, where $I$ is the ideal of $R\langle
\ol{E}; \tau,\tilde{\delta}\rangle$ generated by the idempotents
$q_i:=p_i-\sum _{e\in s^{-1}(i)}e\ol{e}$ for $i\notin \Si (E)$. By
\cite[Remark 2.14]{AB2}, we get that the map
$$R\Sigma_1^{-1}=R\langle \ol{E}; \tau,\tilde{\delta}\rangle /I\longrightarrow
 (P^{{\rm rat}}_{K_r}(E))\langle \ol{E};\tau ,\tilde{\delta}\rangle/I_2=Q_{K_r}(E)$$
is injective, and the image of this map is clearly
$Q_{\mathbf{K}_r}(E;\mathbf{H}_r)$. So we get an isomorphism
$R\Sigma_1^{-1}\cong Q_{\mathbf{K}_r}(E;\mathbf{H}_r)$, which
combined with the isomorphism $R\cong
P_{\mathbf{K}_r}(E;\mathbf{H}_r)\Sigma ^{-1}$ established in Theorem
\ref{rat-locali} gives $Q_{\mathbf{K}_r}(E;\mathbf{H}_r)\cong
(P_{\mathbf{K}_r}(E;\mathbf{H}_r))(\Sigma \cup \Sigma _1)^{-1}$. By
a result of Bergman and Dicks \cite{BD} any universal localization
of a hereditary ring is hereditary, thus we get that both
$P^{\text{rat}}_{\mathbf{K}_r}(E;\mathbf{H}_r)$ and
$Q_{\mathbf{K}_r}(E;\mathbf{H}_r)$ are hereditary rings. Since
$P^{\text{rat}}_{\mathbf{K}_r}(E;\mathbf{H}_r)$ is hereditary,
closed under inversion in $P_{K_r}((E))$ (by Theorem
\ref{rat-locali}), and closed under all the right transductions
$\tilde{\delta}_e$, for $e\in E^1$, the proof of \cite[Theorem
2.16]{AB2} gives that $Q_{\mathbf{K}_r}(E;\mathbf{H}_r)$ is von
Neumann regular and that every finitely generated projective is
induced from $P^{{\rm rat}}_{\mathbf{K}_r}(E;\mathbf{H}_r)$.

This concludes the proof of the theorem.
\end{proof}

\begin{remark}
\label{no-invneed} Theorem 2.16 in \cite{AB2} is stated for a
subalgebra $R$ of $P_K((E))$ which is closed under all left and
right transductions (and which is inversion closed in $P_K((E))$).
However the invariance under right transductions is only used in the
proof of that result to ensure that the ring $R$ is left
semihereditary. Since we are using the opposite notation concerning
(CK1) and (CK2), the above hypothesis translates in our setting into
the condition that $P_{\mathbf{K}_r}(E;\mathbf{H}_r)$ and $P^{{\rm
rat}}_{\mathbf{K}_r}(E;\mathbf{H}_r)$ should be invariant under all
{\it left} transductions, which is not true in general as we
observed above. We overcome this problem by the use of the result of
Bergman and Dicks (\cite{BD}), which guarantees that $
P_{\mathbf{K}_r}(E;\mathbf{H}_r)$  and $P^{{\rm
rat}}_{\mathbf{K}_r}(E;\mathbf{H}_r)$ are indeed right and left
hereditary (see the proof of Theorem \ref{hereditarymixed}).
\end{remark}

Define the mixed Leavitt path algebra
$L_{\mathbf{K}_r}(E;\mathbf{H}_r)$ as the universal localization of
$P_{\mathbf{K}_r}(E;\mathbf{H}_r)$ with respect to the set
$\Sigma_1$. Let $M(E)$ be the abelian monoid with generators $E^0$
and relations given by $v=\sum _{e\in s^{-1}(v)} r(e)$, see
\cite{AMP} and \cite{APW}.

\begin{theor}
\label{isomorphs} With the above notation, we have natural
isomorphisms
$$M(E)\cong  \mon{L_{\mathbf{K}_r}(E_r;\mathbf{H}_r)}\cong
\mon{Q_{\mathbf{K}_r}(E_r;\mathbf{H}_r)}.$$
\end{theor}

\begin{proof}
The proof that $M(E)\cong \mon{L_{\mathbf{K}_r}(E_r;\mathbf{H}_r)}$
follows as an application of Bergman's results \cite{Bergman}, as in
\cite[Theorem 3.5]{AMP}.

Note that $R:=P^{\rm rat}_{\mathbf{K}_r}(E;\mathbf{H}_r)$ is
semiperfect. Thus we get $\mon{R}\cong (\Z^+)^{|E_0|}$ in the
natural way, that is the generators of $\mon{R}$ correspond to the
projective modules $p_vR$ for $v\in E^0$. By Theorem
\ref{hereditarymixed}, we get that the natural map $M(E)\to
\mon{Q_{\mathbf{K}_r}(E_r;\mathbf{H}_r)}$ is surjective. To show
injectivity observe that we have
\begin{equation*}
 M(E)\cong \mon{Q_{K_0}(E)}\longrightarrow
\mon{Q_{\mathbf{K}_r}(E_r;\mathbf{H}_r)}\longrightarrow
\mon{Q_{K_r}(E)}\cong M(E) ,
\end{equation*}
and that the composition of the maps above is the identity. It
follows that the map $M(E)\to
\mon{Q_{\mathbf{K}_r}(E_r;\mathbf{H}_r)}$ is injective and so it
must be a monoid isomorphism.
\end{proof}

\end{document}